\documentclass[12pt]{amsart}
\usepackage{amssymb,hyperref, comment}
\usepackage[all]{xy}

\newtheorem{theorem}{Theorem}[section]
\newtheorem{definition}[theorem]{Definition}

\newtheorem{proposition}[theorem]{Proposition}

\newtheorem{lemma}[theorem]{Lemma}

\newtheorem{remark}[theorem]{Remark}

\begin{document}

\title[Quantum isometry groups]{Quantum isometry groups of duals of free powers of cyclic groups}

\author{Teodor Banica}
\address{T.B.: Department of Mathematics, Cergy-Pontoise University, 95000 Cergy-Pontoise, France. {\tt teodor.banica@u-cergy.fr}}

\author{Adam Skalski}
\address{A.S.: Institute of Mathematics of the Polish Academy of Sciences, ul. \'Sniadeckich 8, 00-956 Warszawa, Poland. {\tt a.skalski@impan.pl}}


\begin{abstract}
We study the quantum isometry groups $G^+(\widehat{\Gamma})$ of the noncommutative Riemannian manifolds associated
to discrete group duals $\widehat{\Gamma}$. The basic representation theory problem is to compute the law of the
main character $\chi:G^+(\widehat{\Gamma})\to\mathbb C$, and our main result here is as follows: for
$\Gamma=\mathbb Z_s^{*n}$, with $s\geq 5$ and $n\geq 2$, the variable $\chi/2$ follows the compound free Poisson
law $\pi_{\underline{\varepsilon}/2}$, where $\varepsilon$ is the uniform measure on the $s$-th roots of unity,
and $\varepsilon\to\underline{\varepsilon}$ is the canonical projection map from complex to real measures. We
discuss as well a number of technical versions of this result, notably with the construction of a new quantum
group, which appears as a ``representation-theoretic limit'', at $s=\infty$.
\end{abstract}

\maketitle

\section*{Introduction}

The notion of compact quantum group was introduced by Woronowicz in \cite{wo1}, \cite{wo2}. Woronowicz's
formalism, which is both quite general, and remarkably easy to handle, allowed Wang to construct in \cite{wa1},
\cite{wa2} a number of universal quantum groups, namely the free analogues of $O_n,U_n,S_n$. The next step,
developed by Bichon in \cite{bic} and by the first-named author in \cite{ban}, was the construction of quantum
automorphism groups of various discrete structures (finite graphs, finite metric spaces). This allowed a
considerable extension of Wang's original list of free quantum groups, notably with a free analogue of the
hyperoctahedral group $H_n$ \cite{bbc}, and a number of technical versions of it \cite{bb+}, \cite{ez1},
\cite{bsp}.

Another important generalization of Wang's original constructions comes from the work of Goswami \cite{gos}. The
idea is that Connes' noncommutative geometry theory \cite{con}, \cite{cma} has shown that a number of important
situations, mainly coming from particle physics or number theory, are best described by a certain spectral triple
$X=(A,H,D)$. It is therefore natural to ask for the computation of the quantum isometry group $G^+(X)$ of such a
spectral triple, constructed in \cite{gos}. The theory here has been quickly developed in the last few years,
first with a number of foundational papers by Bhowmick and Goswami \cite{bg1}, \cite{bg2}, \cite{bg3}. The earlier
quantum automorphism constructions in \cite{ban}, \cite{bic} were reformulated, unified and generalized in
\cite{bgs}, by using the discrete spectral triples constructed by Christensen and Ivan in \cite{civ}. Another
important computation, with several potential applications, is that of the quantum isometry group of Connes'
spectral triple of the standard model: this was recently done by Bhowmick, D'Andrea and Dabrowski \cite{bdd}.

One promising direction in view of the general understanding of the quantum isometry groups is that of the
explicit computation of $G^+(X)$, in the case where $X=\widehat{\Gamma}$ is a discrete group dual: indeed, a
number of very concrete tools, coming from combinatorics, subfactors, or free probability, are available here. The
general theory, as well as a number of examples, including an explicit presentation result in the free group case
$\Gamma=F_n$, were worked out in \cite{bhs}. This latter free group case was studied in much detail in \cite{bsk},
where a link with the hyperoctahedral quantum group in \cite{bbc}, with the Fuss-Catalan combinatorics of Bisch
and Jones \cite{bjo}, and with the free spheres constructed in \cite{bgo} was obtained. More generally, it was
shown in \cite{bsk} that the quantum isometry group $G^+(\widehat{F}_n)$ belongs in fact to a general class of
``two-parameter quantum symmetry groups'', and this makes a link with the recent classification program for the
easy quantum groups, started in \cite{ez1}, \cite{bsp}.

The purpose of the present paper is to push one step forward the previous combinatorial and probabilistic
considerations in \cite{bsk}, \cite{bhs}. These considerations basically concern the computation of the law of the
main character, and our main result here is as follows.

\medskip

\noindent {\bf Theorem.} {\em For $\Gamma=\mathbb Z_s^{*n}$ with $s\geq 5$ and $n\geq 2$, the half of the main character of $G^+(\widehat{\Gamma})$ follows the compound free Poisson law $\pi_{\underline{\varepsilon}/2}$, where $\varepsilon$ is the uniform measure on the $s$-roots of unity, and $\varepsilon\to\underline{\varepsilon}$ is the canonical projection map from complex to real measures.}

\medskip

The proof of this result uses the general diagrammatic methods in \cite{bb+}, \cite{bsk}, Voiculescu's
$R$-transform from \cite{voi}, \cite{vdn} and Speicher's notion of free cumulants from \cite{nsp}, \cite{sp1}.

As in the previous paper \cite{bb+}, we end up with a measure which is a compound free Poisson law, in the sense
of \cite{hpe}, \cite{sp2}. However, the situation here is more complicated, combinatorially speaking, because the
projection map $\varepsilon\to\underline{\varepsilon}$ appears.

The above theorem holds in fact as well at $s=3$. However, at $s=2,4$ the situation is quite different, and we
will present here some results in this direction. The main point here is that the case $s=4$ is truly special, as
already observed in \cite{ban}, \cite{bhs}.

The limiting case $s\rightarrow \infty$ is quite interesting as well, because the quantum groups in the above
theorem ``converge'', in a certain representation theory sense, to a compact quantum group which is different
from the quantum isometry group of the dual of the free group $F_n=\mathbb Z^{*n}$ studied in \cite{bsk}. The
construction of this new quantum group, for which we refer to section 6 below, will be actually our second main
result in this paper.

Finally, let us mention that there are many similarities with the ``easy quantum group'' results in \cite{bb+},
\cite{ez1}, \cite{ez2}. One can expect that some further generalizations of the above theorem might provide some
answers to the questions raised in \cite{ez1}. We do not know if it is so, but we will make some comments in this
direction, at the end of the paper.

The paper is organized as follows: 1 is a preliminary section, in 2 we discuss the general properties of
$G^+(\widehat{\mathbb Z_s^{*n}})$, and in 3-5 we compute diagrams, discuss the compound free Poisson laws, and
prove the above theorem. The final sections, 6-7, contain some technical versions of that theorem, and a few
concluding remarks.

\subsection*{Acknowledgements}

T.B. would like to thank P. Hajac and the IMPAN for the warm hospitality, during a visit in October 2010, when
part of this work was done. The work of T.B. was supported by the ANR grants ``Galoisint'' and ``Granma''. We also want to thank the referees for the comments that have led to an improvement of the paper.

\section{Quantum isometries}

The natural framework for the study of noncommutative objects like $\widehat{\Gamma}$ is Connes' noncommutative
geometry \cite{con}, \cite{cma}, where the basic definition is as follows.

\begin{definition}
A compact spectral triple $(A,H,D)$ consists of the following:
\begin{enumerate}
\item  a unital $C^*$-algebra $A$;

\item  a Hilbert space $H$, on which $A$ acts;

\item an unbounded self-adjoint operator $D$ on $H$, with compact resolvents, such that $[D,x]$ has a bounded
extension, for any $x$ in a dense $*$-subalgebra ${\mathcal A}\subset A$.
\end{enumerate}
\end{definition}

This definition is of course over-simplified, as to best fit with the purposes of the present paper. We refer to
\cite{con}, \cite{cma} for the precise formulation of the axioms.

Now let $\Gamma=<g_1,\ldots,g_m>$ be a discrete group, with the generating set $S=\{g_1,\ldots,g_m\}$ assumed to
satisfy $S=S^{-1}$ and $1\notin S$. The generating set $S$ is of course taken to be a set without repetitions. In
most of the examples below, $m$ will be actually minimal.

The length on $\Gamma$ is given by $l(g)=\min\{r\in\mathbb N|\exists\,h_1\ldots h_r\in S,g=h_1\ldots h_r\}$, and
the distance on $\Gamma$ is given by $d(g,h)=l(g^{-1}h)$. Observe that $S=\{g\in G|l(g)=1\}$.

Consider the group algebra $\mathcal A=\mathbb C\Gamma$, with involution given by $g^*=g^{-1}$ and antilinearity,
and let $A=C^*(\Gamma)$ be the completion of $\mathcal A$ with respect to the maximal $C^*$-norm.

\begin{definition}
Associated to any discrete group $\Gamma=<g_1,\ldots,g_m>$ as above is a compact spectral triple
$\widehat{\Gamma}=(A,H,D)$, as follows:
\begin{enumerate}
\item $A=C^*(\Gamma)$ is the full group algebra of $\Gamma$.

\item $H=l^2(\Gamma)$, with $A$ acting on it by $g\delta_h=\delta_{gh}$.

\item $D(\delta_g)=l(g)\delta_g$, on the standard basis of $H$.
\end{enumerate}
\end{definition}

We can think of $\widehat{\Gamma}$ as being a ``noncommutative Riemannian manifold'', and then consider the
quantum isometry group of orientation preserving isometries $G^+(\widehat{\Gamma})$, introduced by Bhowmick and Goswami in Section 2 of \cite{bg2}.

We recall that the standard trace of $\mathbb C\Gamma\subset C^*(\Gamma)$ is given by $tr(g)=\delta_{g1}$ and
also identify in a natural way $\mathbb C\Gamma$ with a vector subspace of $l^2(\Gamma)$.

\begin{proposition}
$G^+(\widehat{\Gamma})$ is the universal compact quantum group acting on $C^*(\Gamma)$, such that  all the
eigenspaces of $D$  and the standard trace on $\mathbb C \Gamma$  are invariant under this action.
\end{proposition}

\begin{proof}
This is more or less a consequence of a simpler reformulation of the general definition of the quantum isometry groups of orientation preserving isometries given in \cite{bg2}, in the case where the representation of $A$ on $H$ allows a cyclic and separating vector which spans the kernel of $D$ and satisfies natural conditions
with respect to the `smooth algebra' $\mathcal{A}$ (Corollary 2.27 of \cite{bg2}).
The conditions above are satisfied in the
case of manifolds coming from group duals; this, together with the consequences for the computations of $G^+(\widehat{\Gamma})$, is discussed in detail  in Section 2 of \cite{bhs}. Note that the fact that the eigenspaces of $D$ are preserved by the action $\alpha$
implies that $\alpha:\mathbb C \Gamma \to \mathbb C \Gamma  \otimes_{alg} C(G^+(\widehat{\Gamma}))$, so that the
trace preservation condition makes sense.
\end{proof}

Let us look now at the eigenspaces of $D$. There is one such eigenspace for each positive integer $r\in\mathbb N$,
namely the span of the set $\Gamma_r=\{g\in\Gamma|l(g)=r\}$ of words of length $r$.

\begin{proposition}
For any $r\in\mathbb N$, $G^+(\widehat{\Gamma})$ has a unitary representation on $span(\Gamma_r)$, the span of
words of $\Gamma$ of length $r$. Moreover, the representation at $r=1$ is faithful.
\end{proposition}

\begin{proof}
We know from Proposition 1.3 that $G^+(\widehat{\Gamma})$ acts on $span(\Gamma_r)$, and the invariance of the
trace shows that the corresponding representation is unitary. The faithfulness property at $r=1$ follows from the
universal property of $G^+(\widehat{\Gamma})$. See \cite{bhs}.
\end{proof}

The above two statements make it clear how to explicitly construct $G^+(\widehat{\Gamma})$: since this quantum
group has to act faithfully on the span of $\Gamma_1=S$, we just have to consider the universal quantum group
acting on $span(S)$, and then divide by a suitable ideal.

So, let $A_u(m)$ be the universal $C^*$-algebra generated by the entries of a $m\times m$ matrix $u$, such that
both $u=(u_{ij})$ and $u^t=(u_{ji})$ are unitaries. This is a Hopf $C^*$-algebra in the sense of Woronowicz
\cite{wo1}, with comultiplication $\Delta(u_{ij})=\Sigma u_{ik}\otimes u_{kj}$, counit
$\varepsilon(u_{ij})=\delta_{ij}$, and antipode $S(u_{ij})=u_{ji}^*$. Observe that we have $S^2=id$. See Wang
\cite{wa1}.

According to the general results of Woronowicz in \cite{wo1}, we can think of $A_u(m)$ as being the algebra of
continuous functions on a certain compact quantum group. We denote this quantum group by $U_m^+$. That is, $U_m^+$
is the abstract object given by $A_u(m)=C(U_m^+)$.

\begin{theorem}
$G^+(\widehat{\Gamma})$ is the subgroup of $U_m^+$ presented by the following relations:
\begin{enumerate}
\item Those making $\alpha(g_i)=\Sigma g_j\otimes u_{ji}$ a morphism of $*$-algebras.

\item Those making $span(\Gamma_r)$ invariant under $\alpha$, for any $r\geq 2$.
\end{enumerate}
\end{theorem}

\begin{proof}
This result is from \cite{bhs}, where it appears in a slightly different formulation. We present below an
explanation of the present formulation, along with a short proof.

We know from Proposition 1.3 and Proposition 1.4 that the quantum $G^+(\widehat{\Gamma})$ acts faithfully on
$span(S)$, and that the corresponding representation is unitary. In terms of Hopf algebras, this means that we
have a coaction map $\alpha:span(S)\to span(S)\otimes_{alg} C(G^+(\widehat{\Gamma}))$, and that if we denote this
map by $\alpha(g_i)=\Sigma g_j\otimes u_{ji}$, then the corepresentation $u=(u_{ij})$ is unitary, and its
coefficients $u_{ij}$ generate the $C^*$-algebra $C(G^+(\widehat{\Gamma}))$.

So, we obtain in this way a surjective morphism of Hopf $C^*$-algebras $\Phi:A_u(m)\to C(G^+(\widehat{\Gamma}))$,
which corresponds to an embedding of compact quantum groups $G^+(\widehat{\Gamma})\subset U_m^+$.

Let us try to understand now what the kernel of $\Phi$ is. According to the universal property of $\Gamma$, from Proposition 1.3, and to Woronowicz's Tannakian results in \cite{wo2}, this kernel should be exactly the ideal generated by the following relations:

(1) Those making $\alpha$ a morphism of $*$-algebras. This means that for any $h_i,k_i\in S$ satisfying $h_1\ldots h_p=k_1\ldots k_q$ we write the equality $\alpha(h_1\ldots h_p)=\alpha(k_1\ldots k_q)$ as a formula of type $\Sigma_{g\in\Gamma}g\otimes E_g=\Sigma_{g\in\Gamma}g\otimes F_g$, and $E_g=F_g$, with $g\in\Gamma$, are our relations.

(2) Those making $span(\Gamma_r)$ invariant under $\alpha$, for any $r\geq 2$. Once again, an explanation is
needed here. The idea is that if we denote by $\overline{\Gamma}_r$ the set of words of length $\leq r$, then
$span(\overline{\Gamma}_r)$ is invariant under $\alpha$. Now since we have $\Gamma_r\subset\overline{\Gamma}_r$,
the invariance condition for the subspaces $span(\Gamma_r)\subset span(\overline{\Gamma}_r)$ corresponds indeed to
certain relations.

Note that the condition in (2) implies in particular that the action $\alpha$ preserves the trace (as
$tr(\Gamma_r)=\{0\}$ for $r\neq 0$), and the proof is completed.
\end{proof}

Note that in all the cases computed so far in \cite{bhs} and \cite{bsk} the relations in (2) actually were a
consequence of these imposed by (1) -- this will also be the case in this paper. We do not know if this is true
for arbitrary $\Gamma$.

\section{Free products}

In this section we discuss the presentation of the quantum group $G^+(\widehat{\Gamma})$, in the case
$\Gamma=\mathbb Z_s^{*n}$. We will usually assume $s\in\{2,3,\ldots,\infty\}$, with the convention $\mathbb
Z_\infty=\mathbb Z$.

Let us first recall a basic result from \cite{bhs}, solving the problem at $n=1$.

\begin{proposition}
The quantum groups $G^+(\widehat{\mathbb Z_s})$ are as follows:
\begin{enumerate}
\item At $s=\infty$ we have $\widehat{\mathbb Z_s}=\mathbb T$, and $G^+(\mathbb T)=O_2=\mathbb T\rtimes\mathbb Z_2$.

\item More generally, at any $s\neq 4$ we have $G^+(\widehat{\mathbb Z_s})=\mathbb Z_s\rtimes\mathbb Z_2$.

\item At $s=4$ the quantum group $G^+(\widehat{\mathbb Z_s})$ is non-classical, and infinite.
\end{enumerate}
\end{proposition}

\begin{proof}
These results can be deduced from Theorem 1.5, the idea being that for $s\neq 4$, the relations there make the
coefficients $u_{ij}$ commute.  See \cite{bhs}.
\end{proof}

In the general case now $n\in\mathbb N$, let $g_1,\ldots,g_n$ be the standard generators of $\mathbb Z_s^{*n}$, and endow this group with the generating set $S=\{g_{i1},g_{i2}\}$, where $g_{i1}=g_i,g_{i2}=g_i^{-1}$.

The fundamental coaction is denoted as follows:
$$\alpha(g_{ia})=\sum_{jb}g_{jb}\otimes u_{jb,ia}$$

The simplest case is when $s=\infty$. Here our group $\Gamma=\mathbb Z_s^{*n}$ is the free group $F_n$, and we have the following result, from \cite{bsk}.

\begin{proposition}
$H_{n0}^+=G^+(\widehat{F_n})$ is the subgroup of $U_{2n}^+$ presented by the following relations, between the generators denoted $u_{ia,jb}$ with $i,j=1,\ldots,n$ and $a,b=1,2$:
\begin{enumerate}
\item The entries $u_{ia,jb}$ are partial isometries.

\item We have $u_{ia,jb}^*=u_{i\bar{a},j\bar{b}}$, for any $i,j,a,b$.
\end{enumerate}
\end{proposition}

\begin{proof}
This is done in \cite{bsk}, the idea being as follows.

First, we know from Theorem 1.5 that $\alpha$ must preserve the involution, i.e. that we must have $\alpha(g_{ia})^*=\alpha(g_{i\bar{a}})$ for any $i,a$, and this gives $u_{ia,jb}^*=u_{i\bar{a},j\bar{b}}$, for any $i,j,a,b$.

Next, the other condition that we get from Theorem 1.5 is:
\begin{eqnarray*}
\alpha(g_{ia})\alpha(g_{i\bar{a}})=1\otimes 1
&\iff&\sum_{jbkc}g_{jb}g_{kc}\otimes u_{jb,ia}u_{kc,i\bar{a}}=1\otimes 1\\
&\iff&\sum_{jbkc}g_{jb}g_{k\bar{c}}\otimes u_{jb,ia}u_{kc,ia}^*=1\otimes 1
\end{eqnarray*}

Thus we must have $u_{jb,ia}u_{kc,ia}^*=0$ for $(jb)\neq (kc)$, and $\sum_{jb}u_{jb,ia}u^*_{jb,ia}=1$. By
multiplying this latter equality at right by $u_{kc,ia}$, we obtain $u_{kc,ia}u_{kc,ia}^*u_{kc,ia}=u_{kc,ia}$, so
$u_{kc,ia}$ must be a partial isometry. This leads to the conclusion in the statement, as the conditions in (2)
in Theorem 1.5 are automatically satisfied, see the proof of Theorem 5.1 in \cite{bhs}.
\end{proof}

In the above statement, the notation $H_{n0}^+$ comes from the fact that the quantum group under consideration is
part of a certain two-parameter series $H_{nm}^+$, which is such that $H_m^+=H_{0m}^+$ is the hyperoctahedral
quantum group in  \cite{bbc}. See \cite{bsk}.

Let us record as well the following useful reformulation of the above result.

\begin{proposition}
$C(H_{n0}^+)$ is the universal algebra generated by variables $u_{ia,jb}$ with $i,j=1,\ldots,n$ and $a,b=1,2$, subject to the following relations:
\begin{enumerate}
\item $u_{ia,jb}$ are partial isometries, satisfying $u_{ia,jb}^*=u_{i\bar{a},j\bar{b}}$.

\item $u_{jb,ia}u_{kc,ia}^*=0$ for $(jb)\neq (kc)$, and $\sum_{jb}u_{jb,ia}u^*_{jb,ia}=1$.
\end{enumerate}
\end{proposition}

\begin{proof}
This is indeed just a technical reformulation of the above result, the idea being that the relations in (2)
correspond to the fact that the matrices $u,u^t$ are unitaries.
\end{proof}

In the general case now, we have the following statements.

\begin{lemma}
For $s\geq 5$ and $n\geq 2$, the quantum group $H_{n0}^{s+}=G^+(\widehat{\mathbb Z_s^{*n}})$ is the subgroup of
$H_{n0}^+$ presented by the relations coming from $\alpha(g_{ia})^s=1\otimes 1$, for any $i,a$.
\end{lemma}

\begin{proof}
First, let us mention the fact that the result holds as well at $s=3$, with the same proof as below, and at
$s=\infty$ too, with the convention that the extra relations do not exist in this case. These special cases will
be discussed in section 6 below.

Consider the standard coaction, $\alpha(g_{ia})=\sum_{jb}g_{jb}\otimes u_{jb,ia}$. Our assumption $s\neq 1,2,4$ shows that the elements $g_{jb}g_{k\bar{c}}$ are distinct and different from $1$ for $(jb)\neq (kc)$, so the computations in the proof of Proposition 2.2 apply, and show that the generators $u_{ia,jb}$ must satisfy the conditions found there. That is, we have indeed $H_{n0}^{s+}\subset H_{n0}^+$.

According now to Theorem 1.5, the conditions which are left are simply those coming from the equalities
$\alpha(g_{ia})^s=1\otimes 1$, for any $i,a$, and we are done. The conditions in (2) in Theorem 1.5 are satisfied
due to the fact that we are dealing with a quantum subgroup of $H^+_{n0}$ and the only extra relations in $\mathbb Z_s^{*n}$ are those of the form $g^s=e$ -- compare with the proof of Theorem 5.1 in \cite{bhs}.
\end{proof}

\begin{theorem}
For $s\geq 5$ and $n\geq 2$, the quantum group $H_{n0}^{s+}=G^+(\widehat{\mathbb Z_s^{*n}})$ is the subgroup of $H_{n0}^+$ presented by the relation $\xi\in Fix(u^{\otimes s})$, where $\xi=\sum_{ia}e_{ia}^{\otimes s}$.
\end{theorem}

\begin{proof}
According to Lemma 2.4, we just have to understand the meaning of the relations coming from the conditions $\alpha(g_{ia})^s=1\otimes 1$, for any $i,a$. We have:
\begin{eqnarray*}
\alpha(g_{ia})^s
&=&\sum_{j_1\ldots j_s}\sum_{b_1\ldots b_s}g_{j_1b_1}\ldots g_{j_sb_s}\otimes u_{j_1b_1,ia}\ldots u_{j_sb_s,ia}\\
&=&\sum_{\gamma}\gamma\otimes\left(\sum_{j_1\ldots j_s}\sum_{b_1\ldots b_s}\delta_{\gamma,g_{j_1b_1}\ldots g_{j_sb_s}}u_{j_1b_1,ia}\ldots u_{j_sb_s,ia}\right)
\end{eqnarray*}

We know from Proposition 2.3 that we have $u_{jb,ia}u_{kc,ia}=0$ for $(jb)\neq (k\bar{c})$, so we can erase from the above formula all the decompositions $\gamma=g_{j_1b_1}\ldots g_{j_sb_s}$ which are not reduced with respect to the rules $g_{ia}g_{i\bar{a}}=1$. So, if we denote by $W_s$ the set of reduced words of length $s$, we have the following formula:
$$\alpha(g_{ia})^s=1\otimes\left(\sum_{jb}u_{jb,ia}^s\right)+\sum_{\gamma\in W_s}\gamma\otimes \left(\sum_{j_1\ldots j_s}\sum_{b_1\ldots b_s}\delta_{\gamma,g_{j_1b_1}\ldots g_{j_sb_s}}u_{j_1b_1,ia}\ldots u_{j_sb_s,ia}\right)$$

Let us look now at an equality of type $\gamma=\gamma'$, with $\gamma,\gamma'\in W_s$. This can happen only in the situation $g_{jb}^pg_{kc}^q=g_{j\bar{b}}^{s-p}g_{k\bar{c}}^{s-q}$, with $j\neq k$ and $p+q=s$. So, if we denote by $W_s'\subset W_s$ the set of words which are not of this type, we have the following formula:
\begin{eqnarray*}
\alpha(g_{ia})^s
&=&1\otimes\left(\sum_{jb}u_{jb,ia}^s\right)+\sum_{j\neq k}\sum_c\sum_{s=p+q}g_{j1}^pg_{kc}^q\otimes (u_{j1,ia}^pu_{kc,ia}^q+u_{j2,ia}^{s-p}u_{k\bar{c},ia}^{s-q})\\
&&+\sum_{g_{j_1b_1}\ldots g_{j_sb_s}\in W_s'}g_{j_1b_1}\ldots g_{j_sb_s}\otimes u_{j_1b_1,ia}\ldots u_{j_sb_s,ia}
\end{eqnarray*}

Thus the equality $\alpha(g_{ia})^s=1\otimes 1$ is equivalent to the following conditions:
$$\begin{cases}
\sum_{jb}u_{jb,ia}^s=1&\forall\,i,\forall\,a\\
u_{j1,ia}^pu_{kc,ia}^q+u_{j2,ia}^{s-p}u_{k\bar{c},ia}^{s-q}=0&\forall\,j\neq k,\forall\,c,\forall\,s=p+q\\
u_{j_1b_1,ia}\ldots u_{j_sb_s,ia}=0&\forall\,g_{j_1b_1}\ldots g_{j_sb_s}\in W_s'
\end{cases}\quad\quad (1)$$

Let us look now at the condition $\xi\in Fix(u^{\otimes s})$. We have:
\begin{eqnarray*}
u^{\otimes s}(\xi\otimes 1)
&=&\left(\sum_{i_rj_ra_rb_r}e_{j_1b_1,i_1a_1}\otimes\ldots\otimes e_{j_sb_s,i_sa_s}\otimes u_{j_1b_1,i_1a_1}\ldots u_{j_sb_s,i_sa_s}\right)\left(\sum_{ia}e_{ia}^{\otimes s}\otimes 1\right)\\
&=&\sum_{j_1\ldots j_s}\sum_{b_1\ldots b_s}e_{j_1b_1}\otimes\ldots e_{j_sb_s}\otimes\left(\sum_{ia}u_{j_1b_1,ia}\ldots u_{j_sb_s,ia}\right)
\end{eqnarray*}

By using now the formula $u_{jb,ia}u_{kc,ia}=0$, valid for any $(jb)\neq (k\bar{c})$, we get:
\begin{eqnarray*}
u^{\otimes s}(\xi\otimes 1)
&=&\sum_{jb}e_{jb}^{\otimes s}\otimes\left(\sum_{ia}u_{jb,ia}^s\right)\\
&&+\sum_{g_{j_1b_1}\ldots g_{j_sb_s}\in W_s}e_{j_1b_1}\otimes\ldots e_{j_sb_s}\otimes\left(\sum_{ia}u_{j_1b_1,ia}\ldots u_{j_sb_s,ia}\right)
\end{eqnarray*}

Thus the condition $\xi\in Fix(u^{\otimes s})$ is equivalent to the following conditions:
$$\begin{cases}
\sum_{jb}u_{jb,ia}^s=1&\forall\,i,\forall\,a\\
\sum_{ia}u_{j_1b_1,ia}\ldots u_{j_sb_s,ia}=0&\forall\,g_{j_1b_1}\ldots g_{j_sb_s}\in W_s
\end{cases}\quad\quad (2)$$

Let us prove now that we have $(1)\iff (2)$. We begin with $(1)\implies (2)$. Here all the assertions are clear, except perhaps for the equality $\sum_{ia}u_{j_1b_1,ia}\ldots u_{j_sb_s,ia}=0$, with $g_{j_1b_1}\ldots g_{j_sb_s}\in W_s-W_s'$. Now by recalling the definition of $W_s,W_s'$, our word $g_{j_1b_1}\ldots g_{j_sb_s}$ must be of the form $g_{jb}^pg_{kc}^q$ with $j\neq k$ and $p+q=s$, and we must prove that we have $\sum_{ia}u_{jb,ia}^pu_{kc,ia}^q=0$ in this situation. For this purpose, we use the middle condition in (1), namely $u_{j1,ia}^pu_{kc,ia}^q+u_{j2,ia}^{s-p}u_{k\bar{c},ia}^{s-q}=0$. By multiplying by $u_{j1,ia}^*$ we obtain $u_{j1,ia}^*u_{j1,ia}^pu_{kc,ia}^q=0$. By multiplying by $u_{j1,ia}$ we obtain $u_{j1,ia}u_{j1,ia}^*u_{j1,ia}^pu_{kc,ia}^q=0$. But now, since $u_{j1,ia}$ is a partial isometry,  and for each partial isometry $v$ we have $vv^*v=v$, we can remove the $u_{j1,ia}u_{j1,ia}^*$ term, and we are left with $u_{j1,ia}^pu_{kc,ia}^q=0$, and this gives the result, since all terms in the sum $\sum_{ia}u_{jb,ia}^pu_{kc,ia}^q$ are now equal to $0$.

In the other direction now, $(2)\implies (1)$, all the assertions are clear as well: indeed, the third relations
in (1) can be obtained from the second relations in (2) by multiplying by $u_{jb,ia}$, known to be a partial
isometry, and the second relations in (1) can be obtained from the second relations in (2), for the words of type
$g_{j_1b_1}\ldots g_{j_sb_s}\in W_s-W_s'$.
\end{proof}

\section{Noncrossing partitions}

In this section we study the representation theory of $G^+(\widehat{\mathbb Z_s^{*n}})$. According to a well-known principle, going back to Woronowicz's  work in \cite{wo2}, most of the representation theory invariants are encoded in the structure of the spaces $Hom(u^{\otimes k},u^{\otimes l})$, which form a tensor category. So, it is this tensor category that we want to compute.

In addition, another well-known principle, once again going back to the considerations in \cite{wo2}, tells us
that we should look for an explicit basis for the elements of $Hom(u^{\otimes k},u^{\otimes l})$, indexed by
``diagrams''. This latter principle has been proved to be fruitful in a number of situations, see e.g. \cite{ez1},
\cite{bsk}, and we will prove that it applies once again here. Already in \cite{bsk} we noticed that due to the
fact that the quantum isometry groups we study have non-selfadjoint entries in the fundamental representations,
it is natural to consider diagrams with vertices of two different colors.

So, let us first construct a candidate for such a basis of $Hom(u^{\otimes k},u^{\otimes l})$.

\begin{definition}
We let $D_s(k,l)$ be the set of noncrossing partitions between $k$ upper points and $l$ lower points, with each
leg colored black or white, with the following rules:
\begin{enumerate}
\item In each block, the signed number of black legs equals the signed number of white legs, modulo $s$ (the sign being $+$ for an upper leg, and $-$ for a lower leg).

\item We agree to identify any colored partition with all the colored partitions obtained by  interchanging the two colors, independently in each block.
\end{enumerate}
\end{definition}

In this definition the parameter $s$ ranges in the set $\{1,2,3,\ldots,\infty\}$, with the usual convention that $a=b$ modulo $\infty$ means $a=b$. The condition in (1) means to have the equality $b_{up}-b_{down}=w_{up}-w_{down}$ modulo $s$, for each block of our colored partition, where $b_{up},b_{down},w_{up},w_{down}$ count the black/white, up/down legs of the block.

Observe that we have inclusions $D_\infty(k,l)\subset D_s(k,l)\subset D_1(k,l)$, for any $s$. More generally, we have an inclusion $D_s(k,l)\subset D_t(k,l)$, whenever $t|s$.

Here are a few examples of such sets of partitions.

\begin{proposition}
At $s=1,2,\infty$, the sets $D_s(k,l)$ are as follows:
\begin{enumerate}
\item $D_1(k,l)$ is the set of noncrossing partitions $NC(k,l)$, with the legs arbitrarily bicolored black or white, taken modulo the block-interchanging of colors.

\item $D_2(k,l)$ is the set of noncrossing partitions with even blocks $NC_{even}(k,l)$, with the legs arbitrarily bicolored black or white, modulo the block-interchanging of colors.

\item $D_\infty(k,l)$ is once again $NC_{even}(k,l)$, this time colored such that the signed numbers of black and white legs are the same, modulo the block-interchanging of colors.
\end{enumerate}
\end{proposition}

\begin{proof}
This is clear from the fact that the numbers $b_{up},b_{down},w_{up},w_{down}$, which must satisfy $b_{up}-b_{down}=w_{up}-w_{down}$ modulo $s$, sum up to the size of the block.
\end{proof}

Consider now the Hilbert space $H=\mathbb C^{2n}$, with standard basis denoted $\{e_I\}=\{e_{ia}\}$. Associated to any partition $\pi\in D_s(k,l)$ is a linear map $T_\pi:H^{\otimes k}\to H^{\otimes l}$, as follows:
$$T_\pi(e_{I_1}\otimes\ldots\otimes e_{I_k})=\sum_{J_1\ldots J_l}\delta_\pi\begin{pmatrix}I_1&\ldots&I_k\\ J_1&\ldots&J_l\end{pmatrix}e_{J_1}\otimes\ldots\otimes e_{J_l}$$

Here the symbol $\delta_\pi\in\{0,1\}$ is defined as follows: we put the indices on the points of $\pi$, and we set $\delta_\pi=1$ when for each block there exists an index $I=ia$ such that all black indices are $I=ia$, and all white indices are $\bar{I}=i\bar{a}$. Otherwise, we set $\delta_\pi=0$.

\begin{proposition}
The application $\pi\to T_\pi$ transforms the horizontal concatenation, vertical concatenation and upside-down turning of colored diagrams into the tensor product, composition, and involution of linear maps (with a $2n$ factor for the circles).
\end{proposition}

\begin{proof}
All the assertions are clear from definitions, except perhaps for the last assertion, concerning the value of the circles. So, consider the partition $\cap$, with legs colored $bw$. This partition belongs to $D_s(0,2)$ for any $s$, and the associated linear map is by definition $T_\cap(1)=\sum_{ia}e_{ia}\otimes e_{i\bar{a}}$. Similarly, the partition $\cup$, with legs colored $bw$, belongs to $D_s(2,0)$ for any $s$, and the associated linear map is $T_\cup(e_{ia}\otimes e_{jb})=\delta_{ia,j\bar{b}}$. Now since we have $T_\cup T_\cap(1)=T_\cup(\sum_{ia}e_{ia}\otimes e_{i\bar{a}})=\#\{ia\}=2n$, this gives the result.
\end{proof}

In order to make the link with the quantum groups considered in the previous section, we will need a supplementary category of partitions, constructed in \cite{bsk}:

\begin{definition}
We let $\overline{D}_\infty(k,l)$ be the set of partitions in $NC_{even}(k,l)$, with all legs colored black or white, with the following rules:
\begin{enumerate}
\item In each block, the top and bottom legs either both follow the pattern $bwbw\ldots$, or both follow the pattern $wbwb\ldots$

\item We agree to identify any colored partition with all the colored partitions obtained by interchanging the two colors, independently in each block.
\end{enumerate}
\end{definition}

Once again, we can associate linear maps to such partitions, by the same formula as the one given before
Proposition 3.3, and an analogue of that proposition holds. See \cite{bsk}.

Observe that for any $s$ we have inclusions as follows:
$$\overline{D}_\infty(k,l)\subset D_\infty(k,l)\subset D_s(k,l)\subset D_1(k,l)$$

Here the first inclusion follows by comparing Definition 3.4 with Proposition 3.2 (3), and the other two inclusions were already noticed, before Proposition 3.2.

As we will see in the proof below, the above inclusions correspond, via Woronowicz's Tannakian duality \cite{wo2}, to inclusions between the quantum groups we are interested in.

\begin{theorem}
For the quantum group $G^+(\widehat{\mathbb Z_s^{*n}})$, with $5\leq s<\infty$ and $n\geq 2$, we have:
$$Hom(u^{\otimes k},u^{\otimes l})=span(T_\pi|\pi\in D_s(k,l))$$
In addition, the linear maps $T_\pi$ appearing on the right are linearly independent.
\end{theorem}

\begin{proof}
Let us first mention the fact that the result holds as well at $s=3$. Regarding this case, or more generally all the special cases ($s=2,3,4,\infty$) we refer to section 6 below.

So, let $G$ be the quantum isometry group in the statement.

{\bf Step 1.} We first find a purely combinatorial formulation of the problem.

We know from Theorem 2.5 that $G\subset H_{n0}^+$ appears as the subgroup presented by the relations $\xi\in Fix(u^{\otimes k})$, where $\xi=\sum_{ia}e_{ia}^{\otimes s}$.  We also know from \cite{bsk} that the category of partitions for $H_{n0}^+$ is the above category $\overline{D}_\infty$. Now since we have $\xi=T_{1_s}$, where $1_s$ is the 1-block partition having $s$ legs, we conclude that the category of partitions for $G$ is:
$$D_s'=<\overline{D}_\infty,1_s>$$

Summarizing, we have to prove that we have $D_s=D_s'$. Since the inclusion $D_s'\subset D_s$ is clear from definitions, we are left with proving the inclusion $D_s\subset D_s'$.

{\bf Step 2.} We construct some useful elements in $D_s'=<\overline{D}_\infty,1_s>$.

Consider the 1-block partition $1_s\in NC(0,s)$, with all legs colored black.
By capping it at right, from below, with the partition ${\raisebox{10pt} {\xymatrix@R=0.2pc @C=0.2pc{ *=0{\bullet}  & & *=0{\circ}
\\ & *=0{} \ar@{-}[ur]  \ar@{-}[ul]  & }}  }$, we obtain a 1-block partition in $NC(1,s-1)$, with the upper leg white, and all the lower legs black (in the example $s=5$):
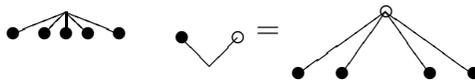
\[ {\raisebox{10pt} {\xymatrix@R=0.4pc @C=0.4pc{  & & *=0{} & &
\\ *=0{\bullet} \ar@{-}[urr] & *=0{\bullet} \ar@{-}[ur] &  *=0{\bullet} \ar@{-}[u]  & *=0{\bullet} \ar@{-}[ul] & *=0{\bullet} \ar@{-}[ull]}}}
\;\;\;\;\;
{\raisebox{0pt} {\xymatrix@R=0.3pc @C=0.3pc{ *=0{\bullet}  & & *=0{\circ}
\\ & *=0{} \ar@{-}[ur]  \ar@{-}[ul]  & }}  }
 =
  {\raisebox{10pt}
 {\xymatrix@R=0.4pc @C=0.8pc{  & & *=0{\circ} & & \\  & & & &
\\ *=0{\bullet} \ar@{-}[uurr] & *=0{\bullet} \ar@{-}[uur] &    & *=0{\bullet} \ar@{-}[uul] & *=0{\bullet} \ar@{-}[uull]}}}
\]
We can repeat this process, and for any $x\in\{1,\ldots,s\}$ we obtain a partition in $NC(x,s-x)$, having all upper legs white, and all the lower legs black. We call these partitions ``forks''.

Consider now once again the 1-block partition $1_s\in NC(0,s)$, with all legs colored black. By capping it from
below in the center with a `colour inverted' fork in $NC(s-2,2)$ (so for $s=5$ it is  ${\raisebox{10pt} {\xymatrix@R=0.2pc @C=0.2pc{*=0{\bullet} & & *=0{\bullet}  & & *=0{\bullet}
\\ &  *=0{\circ} \ar@{-}[ur]  \ar@{-}[ul]&  & *=0{\circ} \ar@{-}[ur]  \ar@{-}[ul]  &  }}  }$) we obtain the partition ${\raisebox{10pt} {\xymatrix@R=0.3pc @C=0.3pc{  & & *=0{} & &
\\ *=0{\bullet} \ar@{-}[urr] & *=0{\circ} \ar@{-}[ur] & & *=0{\circ} \ar@{-}[ul]  & *=0{\bullet} \ar@{-}[ull] }}}$, which we call a ``comb''.

{\bf Step 3.} We prove first that we have $D_\infty\subset D_s'$.

Since $D_\infty$ is generated by its blocks, it is enough to show that any 1-block partition in $D_\infty$ belongs to $D_s'$. Also, by Frobenius duality, which diagramatically corresponds to rotating the legs, we can restrict attention to the partitions having no upper legs. So, what we want to prove is that any 1-block partition $\pi\in D_\infty(0,k)$ is in $D_s'$.

Now recall the definition of $D_\infty$: these are the partitions such that in each block, the signed number of
black legs equals the signed number of white legs. Now since our partition $\pi$ has just one block, and no upper
legs, the condition $\pi\in D_\infty(0,k)$ simply tells us that the number of black legs should equal the number
of white legs.

We proceed by recurrence on the number of legs (note that if $k=2$ then $\pi \in \overline{D}_{\infty}$, so we
can start from $k=4$). First, by using a rotation, we can assume that the last 2 legs at right have different
colors, say $bw$. Depending now on the color of the leg preceding these 2 last legs, we have two cases:

(1) Case $\pi=Rbbw$, with $R$ word in $b,w$. In this case we put at right of $\pi'=Rb$, supposed by recurrence to be in $D_s'$, the colour inverted comb ${\raisebox{10pt} {\xymatrix@R=0.2pc @C=0.2pc{  & & *=0{} & &
\\ *=0{\circ} \ar@{-}[urr] & *=0{\bullet} \ar@{-}[ur] & & *=0{\bullet} \ar@{-}[ul]  & *=0{\circ} \ar@{-}[ull] }}}$, which is as well in $D_s'$ by Step 2, and we cap in the middle with ${\raisebox{10pt} {\xymatrix@R=0.2pc @C=0.2pc{ *=0{\bullet}  & & *=0{\circ}
\\ & *=0{} \ar@{-}[ur]  \ar@{-}[ul]  & }}  }$ from below. We obtain $\pi$, and we are done.

(2) Case $\pi=Rwbw$, with $R$ word in $b,w$. In this case we put at right of $\pi'=Rw$, supposed by recurrence to be in $D_s'$, the partition  ${\raisebox{10pt} {\xymatrix@R=0.2pc @C=0.2pc{  & & *=0{} & &
\\ *=0{\bullet} \ar@{-}[urr] & *=0{\circ} \ar@{-}[ur] & & *=0{\bullet} \ar@{-}[ul]  & *=0{\circ} \ar@{-}[ull] }}}$, which is in $\overline{D}_\infty\subset D_s'$, and we cap in the middle with ${\raisebox{10pt} {\xymatrix@R=0.2pc @C=0.2pc{ *=0{\circ}  & & *=0{\bullet}
\\ & *=0{} \ar@{-}[ur]  \ar@{-}[ul]  & }}  }$ from below. We obtain once again $\pi$, and we are done.

{\bf Step 4.} We prove now that we have indeed $D_s\subset D_s'$.

According to the result in Step 3, it suffices to show that we have $D_s\subset <D_\infty,1_s>$. In order to do this, we proceed as in Step 3: since $D_s$ is generated by its blocks, and by Frobenius duality, it is enough to check that any 1-block partition $\pi\in D_\infty(0,k)$ is in $<D_\infty,1_s>$. That is, we have to show that any 1-block partition $\pi\in D_\infty(0,k)$ whose number of black legs equals the number of white legs modulo $s$ is in $<D_\infty,1_s>$.

Let then  $\pi\in D_s$ be a 1-block element, say in $NC(2p+ks)$, where $p,k \in \mathbb N$ and $p$ denotes the number of black legs in $\pi$. It suffices to observe that it can be always realised as a capping of a 1-block partition $\pi'\in D_{\infty}\cap NC(2p+2ks)$, given by $\pi$ followed by $ks$ black legs, by $k$ copies of our partition $1_s$ with all black legs from below and on the right. A simple example is given below:
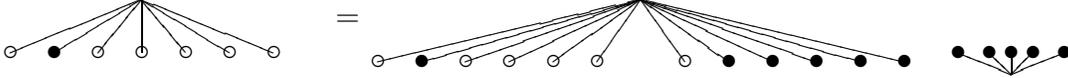
\[
  {\raisebox{10pt}
 {\xymatrix@R=0.4pc @C=0.8pc{  & & &  *=0{} & & && \\   & & & & & &
\\ *=0{\circ} \ar@{-}[uurrr] & *=0{\bullet} \ar@{-}[uurr] &   *=0{\circ} \ar@{-}[uur]   &  *=0{\circ} \ar@{-}[uu] & *=0{\circ} \ar@{-}[uul] & *=0{\circ} \ar@{-}[uull]   & *=0{\circ} \ar@{-}[uulll]  }}}
=
 {\raisebox{10pt}
 {\xymatrix@R=0.4pc @C=0.8pc{  & & & & & &  *=0{}  & & & & & &\\  & & & & & & & & & & & &
\\ *=0{\circ} \ar@{-}[uurrrrrr] & *=0{\bullet} \ar@{-}[uurrrrr] &   *=0{\circ} \ar@{-}[uurrrr]   &  *=0{\circ} \ar@{-}[uurrr] & *=0{\circ} \ar@{-}[uurr] & *=0{\circ} \ar@{-}[uur]  &  & *=0{\circ} \ar@{-}[uul]  &   *=0{\bullet} \ar@{-}[uull] &  *=0{\bullet} \ar@{-}[uulll]  &  *=0{\bullet} \ar@{-}[uullll]  &  *=0{\bullet} \ar@{-}[uulllll]  &  *=0{\bullet} \ar@{-}[uullllll]  }}}
\;\;\;\;
{\raisebox{-10pt} {\xymatrix@R=0.4pc @C=0.4pc{ *=0{\bullet}  & *=0{\bullet}  & *=0{\bullet} & *=0{\bullet} & *=0{\bullet}
\\ & &  *=0{} \ar@{-}[urr] \ar@{-}[u]  \ar@{-}[ull] \ar@{-}[ur]  \ar@{-}[ul] & &  }}  }
\]

{\bf Step 5.} It remains to prove that the maps $T_\pi$ are linearly independent.

The linear independence is a well-known issue, and can be solved by using a standard trick, namely a positivity argument involving a trace, going back to Jones' paper \cite{jon}. The idea here is to solve first the problem in the case $k=l$, where the span of the diagrams has a natural structure of $C^*$-algebra, and then to use Frobenius duality and the natural embeddings $D_s(k,l)\to D_s(k,l+s)$, in order to get the result for any $k,l$.

However, in our case we can simply use some previously known results, as follows.

Since we have $D_s(k,l)\subset D_1(k,l)$, our problem is actually about the diagrams in $D_1$, namely the partitions in $NC$, bicolored, modulo the block-interchanging of colors.

So, let us recall from \cite{bsk} that if $S_{n0}^+\subset H_{n0}^+$ denotes the quantum group presented by the relations making all the standard generators $u_{ia,jb}$ projections, then we have $Hom(u^{\otimes k},u^{\otimes l})=span(T_\pi|\pi\in D_1(k,l))$. We refer as well to \cite{bsk} for the precise relation between $S_{n0}^+$ and the quantum permutation group $S_n^+$, constructed by Wang in \cite{wa2}.

Let us also  recall from \cite{bsk} that we have an isomorphism $S_{n0}^+\simeq H_n^+$, where $H_n^+$ is the hyperoctahedral quantum group constructed in \cite{bbc}, taken with its sudoku representation.

With these results in hand, the linear independence is clear: indeed, the dimension of $Hom(u^{\otimes k},u^{\otimes l})$ is known from \cite{bbc}, and since a basic diagram count, performed in \cite{bsk}, shows that this is exactly the number of diagrams in $D_1(k,l)$, this gives the result.
\end{proof}

As a first comment, there should be as well a third proof for the linear independence, just by showing that the Gram determinant of the linear maps $T_\pi$ is nonzero. Indeed, two key ``Temperley-Lieb'' determinants were computed in \cite{dgg}, \cite{tut}, and, according to the results in \cite{bcu}, these computations can be usually extended to objects of type $D_1$.

Note however that the computation of the Gram determinant for $D_s$ with $s\geq 3$ is probably quite a difficult
task, because the size of the matrix, which will be shown to be a moment of a certain compound Poisson law, is a
quite complicated combinatorial object. So, computing this Gram determinant is a question that we would like to
raise here.

Finally, let us mention that these Gram determinant questions are of particular interest in relation with the Weingarten formula \cite{csn}. See \cite{ez2} for some potential applications.

\section{Poisson laws}

In the remainder of this paper we present a concrete application of Theorem 3.5, namely the computation of the Kesten measure of the discrete quantum group $\Lambda=\widehat{G}$, where $G=G^+(\widehat{\mathbb Z_s^{*n}})$. This measure is by definition the spectral measure of the character of the fundamental representation, $\chi=\Sigma u_{ii}$, with respect to the Haar functional.

For the moment, let us start with some probabilistic preliminaries. It is known from \cite{bb+}, \cite{ez1}, \cite{ez2}, \cite{bsk} that what we can expect to find as Kesten measures for our quantum groups should be some versions of the free Poisson law. So, in this section we will study in detail the free Poisson law, and its ``compound'' versions, first constructed in \cite{hpe}, \cite{sp2}.

\begin{definition}
The Poisson law $p_t$ and the free Poisson law $\pi_t$ are the real probability measures appearing when performing a Poisson limit, respectively a free Poisson limit
$$p_t=\lim_{n\to\infty}\left(\left(1-\frac{1}{n}\right)\delta_0+\frac{1}{n}\delta_t\right)^{* n}\quad
\pi_t=\lim_{n\to\infty}\left(\left(1-\frac{1}{n}\right)\delta_0+\frac{1}{n}\delta_t\right)^{\boxplus n}$$
where $*$ is the usual convolution of measures, and $\boxplus$ is Voiculescu's free convolution \cite{voi}.
\end{definition}

A well-known computation shows that at $t=1$ the Poisson law is $p_1=\frac{1}{e}\sum_{k=1}^\infty\frac{\delta_k}{k!}$, and that the free Poisson law is $\pi_1=\frac{1}{2\pi}\sqrt{1-4x^{-1}}\,dx$. Note that $\pi_1$ is the same as the Marchenko-Pastur law, discovered in a slightly different context in \cite{mpa}. See \cite{nsp}.

More generally, let $\rho$ be a compactly supported positive measure on $\mathbb{R}$. By replacing in the above formulae $\delta_t$ by $\rho$, we obtain measures called compound Poisson/free Poisson laws.

\begin{definition}
Associated to any compactly supported positive measure $\rho$ on $\mathbb{R}$ are the probability measures
$$p_\rho=\lim_{n\to\infty}\left(\left(1-\frac{c}{n}\right)\delta_0+\frac{1}{n}\rho\right)^{* n}\quad
\pi_\rho=\lim_{n\to\infty}\left(\left(1-\frac{c}{n}\right)\delta_0+\frac{1}{n}\rho\right)^{\boxplus n}$$
where $c=mass(\rho)$, called compound Poisson and compound free Poisson laws.
\end{definition}

In what follows we will be interested in the case where $\rho$ is discrete, as is for instance the case for $\rho=\delta_t$ with $t>0$, which produces the Poisson and free Poisson laws.

We recall that the usual convolution operation $*$ is linearized by $\log F$, where $F$ is the Fourier transform. There are several ways of defining and normalizing $F$, and for the purposes of this paper, we will take as definition $F_{\delta_z}(y)=e^{-iyz}$, on the Dirac masses.

We recall also from Voiculescu \cite{voi} that the free convolution operation $\boxplus$ is linearized by the $R$-transform, constructed as follows: first, we let $f(y)=1+m_1y+m_2y^2+\ldots$ be the moment generating function of our measure, so that $G(\xi)=\xi^{-1}f(\xi^{-1})$ is the Cauchy transform; then we set $R(y)=K(y)-y^{-1}$, where $K(y)$ is such that $G(K(y))=y$.

The following result allows one to detect compound Poisson/free Poisson laws.

\begin{lemma}
For $\rho=\sum_{i=1}^sc_i\delta_{z_i}$ with $c_i>0$ and $z_i\in\mathbb R$, we have
$$F_{p_\rho}(y)=\exp\left(\sum_{i=1}^sc_i(e^{-iyz_i}-1)\right)\quad R_{\pi_\rho}(y)=\sum_{i=1}^s\frac{c_iz_i}{1-yz_i}$$
where $F,R$ denote respectively the Fourier transform, and Voiculescu's $R$-transform.
\end{lemma}

\begin{proof}
Let $\mu_n$ be the measure appearing in Definition 4.2, under the convolution signs. In the classical case, we have the following well-known computation:
\begin{eqnarray*}
F_{\mu_n}(y)=\left(1-\frac{c}{n}\right)+\frac{1}{n}\sum_{i=1}^sc_ie^{-iyz_i}
&\implies&F_{\mu_n^{*n}}(y)=\left(\left(1-\frac{c}{n}\right)+\frac{1}{n}\sum_{i=1}^sc_ie^{-iyz_i}\right)^n\\
&\implies&F_{p_\rho}(y)=\exp\left(\sum_{i=1}^sc_i(e^{-iyz_i}-1)\right)
\end{eqnarray*}

In the free case now, we use a similar method. First, we have:
\begin{eqnarray*}
f_{\mu_n}(y)=\left(1-\frac{c}{n}\right)+\frac{1}{n}\sum_{i=1}^s\frac{c_i}{1-z_iy}
&\implies&G_{\mu_n}(\xi)=\left(1-\frac{c}{n}\right)\frac{1}{\xi}+\frac{1}{n}\sum_{i=1}^s\frac{c_i}{\xi-z_i}\\
&\implies&y=\left(1-\frac{c}{n}\right)\frac{1}{K_{\mu_n}(y)}+\frac{1}{n}\sum_{i=1}^s\frac{c_i}{K_{\mu_n}(y)-z_i}
\end{eqnarray*}

Now since $K_{\mu_n}(y)=y^{-1}+R_{\mu_n}(y)=y^{-1}+R/n$, where $R=R_{\mu_n^{\boxplus n}}(y)$, we get:
\begin{eqnarray*}
&&y=\left(1-\frac{c}{n}\right)\frac{1}{y^{-1}+R/n}+\frac{1}{n}\sum_{i=1}^s\frac{c_i}{y^{-1}+R/n-z_i}\\
&\implies&1=\left(1-\frac{c}{n}\right)\frac{1}{1+yR/n}+\frac{1}{n}\sum_{i=1}^s\frac{c_i}{1+yR/n-yz_i}
\end{eqnarray*}

Now multiplying by $n$, rearranging the terms, and letting $n\to\infty$, we get:
\begin{eqnarray*}
\frac{c+yR}{1+yR/n}=\sum_{i=1}^s\frac{c_i}{1+yR/n-yz_i}
&\implies&c+yR_{\pi_\rho}(y)=\sum_{i=1}^s\frac{c_i}{1-yz_i}\\
&\implies&R_{\pi_\rho}(y)=\sum_{i=1}^s\frac{c_iz_i}{1-yz_i}
\end{eqnarray*}

This finishes the proof in the free case, and we are done.
\end{proof}

We have as well the following result, which provides an alternative to Definition 4.2.

\begin{theorem} \label{PoissonR}
For $\rho=\sum_{i=1}^sc_i\delta_{z_i}$ with $c_i>0$ and $z_i\in\mathbb R$, we have
$$p_\rho/\pi_\rho={\rm law}\left(\sum_{i=1}^sz_i\alpha_i\right)$$
where the variables $\alpha_i$ are Poisson/free Poisson$(c_i)$, independent/free.
\end{theorem}

\begin{proof}
Observe first that the result holds for $\rho=\delta_t$, with $t>0$. Also, it was shown in \cite{bb+} that an analogue of this result holds for $\rho$ = uniform measure on the $s$-th roots of unity (note however that the latter measure is not supported on $\mathbb R$).

In the general case, let $\alpha$ be the sum of Poisson/free Poisson variables in the statement. We will show that the Fourier/$R$-transform of $\alpha$ is given by the formulae in Lemma 4.3.

Indeed, by using some well-known Fourier transform formulae, we have:
\begin{eqnarray*}
F_{\alpha_i}(y)=\exp(c_i(e^{-iy}-1))
&\implies&F_{z_i\alpha_i}(y)=\exp(c_i(e^{-iyz_i}-1))\\
&\implies&F_\alpha(y)=\exp\left(\sum_{i=1}^sc_i(e^{-iyz_i}-1)\right)
\end{eqnarray*}

Also, by using some well-known $R$-transform formulae, we have:
\begin{eqnarray*}
R_{\alpha_i}(y)=\frac{c_i}{1-y}
&\implies&R_{z_i\alpha_i}(y)=\frac{c_iz_i}{1-yz_i}\\
&\implies&R_\alpha(y)=\sum_{i=1}^s\frac{c_iz_i}{1-yz_i}
\end{eqnarray*}

Thus we have indeed the same formulae as those in Lemma 4.3, and we are done.
\end{proof}

\begin{remark}
It is tempting to replace in the above results compactly supported positive measures on $\mathbb R$ by their counterparts supported on $\mathbb C$. It is easy to see that the classical parts pass through without any modifications; it is the free case that offers a challenge. The free convolution is defined in terms of the distribution of the sum of two free noncommutative random variables. If the variables in question are selfadjoint (which corresponds to the measures being supported on $\mathbb R$), then so is their sum. There is however no reason why the sum of normal operators should be normal, so, although the formal computations do not change, it is not clear how they should be interpreted. We thank the referee for pointing this out.
\end{remark}

\section{Spectral measures}

Let us go back now to our quantum group problematics. We recall from the beginning of section 3 that most of the representation theory invariants of a quantum group $(G,u)$ are known to be encoded in the structure of the linear spaces $Hom(u^{\otimes k},u^{\otimes l})$.

Here $u$ denotes the fundamental corepresentation of $C(G)$, and we assume that we have $u=\bar{u}$ as corepresentations. Note that this is true for the quantum group $H_{n0}^+$ from \cite{bsk}, \cite{bhs}, hence is true as well for the quantum groups investigated in this paper.

The main representation theory invariant is the dimension of the above spaces. By Woronowicz's general results in \cite{wo1} we have $\dim(Hom(u^{\otimes k},u^{\otimes l}))=h(\chi^{k+l})$, where $\chi=\Sigma u_{ii}$ is the character of $u$, and where $h:C(G)\to\mathbb C$ is the Haar functional.

So, the main problem is to compute the moments of $\chi$. For several reasons, it is actually more convenient to specify directly the law of $\chi$.

\begin{definition}
The spectral measure of a compact quantum group $(G,u)$ with $u=\bar{u}$ is the real probability measure $\mu$ given by $\int\varphi(x)\,d\mu(x)=h(\varphi(\chi))$, where $\chi=\Sigma u_{ii}$ is the character of the fundamental representation, and $h:C(G)\to\mathbb C$ is the Haar functional.
\end{definition}

For a number of other interpretations of this key measure, and a number of motivations for its explicit computation, we refer for instance to \cite{ban}, \cite{ez1}.

We are now in position to state and prove our main result. Let $\rho\to\underline{\rho}$ the projection map from
complex to real measures, given on Dirac masses by $\delta_z\to\delta_{Re(z)}$.

\begin{theorem}
The spectral measure of $G^+(\widehat{\mathbb Z_s^{*n}})$ with $5\leq s<\infty$ and $n\geq 2$ is given by
$$law(\chi/2)=\pi_{\underline{\varepsilon}/2}$$
where $\varepsilon$ is the uniform measure on the $s$-roots of unity.
\end{theorem}

\begin{proof}
According to the general theory developed by Woronowicz in \cite{wo1}, the moments of the character $\chi=\Sigma u_{ii}$ are given by:
$$\int\chi^r=\dim(Fix(u^{\otimes r}))$$

So, we have to count the diagrams in Theorem 3.5. If we denote by $\kappa_r$ the number of ways of coloring an
$r$-block, then the number of diagrams in $D_s(0,r)$ is:
$$m_r=\sum_{\pi\in NC(r)}\prod_{b\in\pi}\kappa_{\# b}$$

According now to Speicher's results in \cite{sp1} (see Proposition 11.4 (3) in \cite{nsp}), we can conclude that the free cumulants of the measure
under consideration are the above numbers $\kappa_r$.

We will write $p=q(s)$ to indicate that $p=q$ modulo $s$. Since for coloring a $r$-block, as in Definition 3.1,
what we have to do is just to pick two numbers $p,q$ satisfying $p+q=r$ and $p=q(s)$, then choose $p$ legs of our
block and color them black, and color the remaining $q$ legs white, remembering that the colors can be also
exchanged, the numbers $\kappa_r$ are given by:
$$\kappa_r=\frac{1}{2}\sum_{p+q=r,\ p=q(s)}\binom{r}{p}$$

Consider now the root of unity $w=e^{2\pi i/s}$. Our claim is that we have:
$$\kappa_r=\frac{1}{2s}\sum_{k=1}^s(w^k+w^{-k})^r$$

Indeed, the equality between the above two quantities can be checked as follows:
\begin{eqnarray*}
\frac{1}{s}\sum_{k=1}^s(w^k+w^{-k})^r
&=&\frac{1}{s}\sum_{k=1}^s\sum_{p=0}^r\binom{r}{p}(w^k)^p(w^{-k})^{r-p}\\
&=&\sum_{p=0}^r\binom{r}{p}\left(\frac{1}{s}\sum_{k=1}^s(w^{2p-r})^k\right)\\
&=&\sum_{p+q=r,\,p=q(s)}\binom{r}{p}
\end{eqnarray*}

We use now the fact, once again originally proved in  Speicher's paper \cite{sp1}, that the free cumulants are the coefficients of the $R$-transform (actually in \cite{nsp} the $R$-transform is first introduced as the power series with the coefficients given by the free cumulants and only later related to the functional equations).  So, by using the above formula, we conclude that the $R$-transform $R_\chi(y)=\kappa_1+\kappa_2y+\kappa_3y^2+\ldots$ of our measure is given by:
$$R_\chi(y)=\frac{1}{2s}\sum_{k=1}^s\frac{w^k+w^{-k}}{1-(w^k+w^{-k})y}$$

Now by using the well-known dilation formula $R_{a\chi}(y)=aR_\chi(ay)$, we obtain:
$$R_{\chi/2}(y)=\frac{1}{4s}\sum_{k=1}^s\frac{w^k+w^{-k}}{1-(w^k+w^{-k})y/2}$$

Consider now the uniform measure on the $s$-th roots of unity, $\varepsilon=\frac{1}{s}\sum_{k=1}^s\delta_{w^k}$.
Its image via the projection map from complex to real measures, $\delta_z\to\delta_{Re(z)}$, is given by:
$$\underline{\varepsilon}=\frac{1}{s}\sum_{k=1}^s\delta_{(w^k+w^{-k})/2}$$

Thus the real measure appearing in the statement is:
$$\underline{\varepsilon}/2=\frac{1}{2s}\sum_{k=1}^s\delta_{(w^k+w^{-k})/2}$$

By using now Lemma 4.3, we have:
$$R_{\pi_{\underline{\varepsilon}}/2}(y)=\frac{1}{4s}\sum_{k=1}^s\frac{w^k+w^{-k}}{1-(w^k+w^{-k})y/2}$$

Thus we have $R_{\chi/2}(y)=R_{\pi_{\underline{\varepsilon}}/2}(y)$, and this gives the conclusion in the statement.
\end{proof}

We can reformulate the above statement in the following way.

\begin{proposition}
For $\Gamma=\mathbb Z_s^{*n}$ with $5\leq s<\infty$ and $n\geq 2$, the main character of the quantum isometry group $G^+(\widehat{\Gamma})$ decomposes as
$$\chi=\sum_{k=1}^s2\cos\left(\frac{2k\pi}{s}\right)\alpha_k$$
where $\alpha_1,\ldots,\alpha_s$ are free Poisson variables of parameter $1/(2s)$, free.
\end{proposition}

\begin{proof}
This follows from Theorem 4.4 and from Theorem 5.2. Indeed, if $\alpha_1,\ldots,\alpha_k$ are free Poisson variables as in the statement, we have:
$$law(\chi/2)=\pi_{\underline{\varepsilon}/2}=\pi_{\left\{\frac{1}{2s}\sum_{k=1}^s\delta_{(w^k+w^{-k})/2}\right\}}=law\left(\sum_{k=1}^s\frac{w^k+w^{-k}}{2}\,\alpha_k\right)$$

Thus $\chi$ has the same law as twice the variable on the right, and we are done.
\end{proof}

As a first observation, the above results remind those previously found in \cite{bb+}, \cite{ez1}, where the free
Poisson law, and its compound versions, also play a central role.

More precisely, for the quantum groups $H_n^s,H_n^{s+},H_n^{(s)}$ considered in \cite{bb+}, \cite{ez1}, the
corresponding asymptotic spectral measures are $p_\varepsilon$, $\pi_\varepsilon$, $\tilde{p_\varepsilon}$
respectively, where $\varepsilon\to\tilde{\varepsilon}$ is the ``squeezing'' operation $\delta_z\to\delta_{|z|}$,
from the complex to the real probability measures.

So, the quantum group $G^+(\widehat{\mathbb Z_s^{*n}})$, while being not exactly ``easy'' in the sense of
\cite{bb+}, \cite{ez1}, seems to belong to the same circle of ideas as the general ``hyperoctahedral series'' of
easy quantum groups, whose existence was conjectured in \cite{ez1}. This suggests that the hyperoctahedral
series, or at least a big part of it, might appear via some simple algebraic manipulations from the quantum
isometry groups of type $G^+(\widehat{\Gamma})$, with $\Gamma$ of the form $\mathbb Z_{s_1}*\ldots*\mathbb
Z_{s_n}$, or perhaps a bit more general. We do not know if it is so.

\section{Special cases}

In this section we discuss the special cases $s=2,3,4,\infty$. We begin with some results at $s=2,3,4$, refining and generalizing those found in \cite{bhs}, at $n=1$.

\begin{proposition}
$G^+(\widehat{\mathbb Z_2^{*n}})$ is the hyperoctahedral quantum group $H_n^+$ from \cite{bbc}.
\end{proposition}

\begin{proof}
Let $g_1,\ldots,g_n$ be the standard generators of $\mathbb Z_2^{*n}$. We write the coaction of the quantum isometry group in the statement as $\alpha(g_i)=\Sigma_jg_j\otimes u_{ji}$. Then, according to the general results in section 1, the relations between the generators $u_{ij}$ are as follows:
\begin{enumerate}
\item First, we have the relations $u_{ij}=u_{ij}^*$ (coming from $g_i=g_i^*$).

\item Second, we have $u_{ji}u_{ki}=0$ for $j\neq k$, and $\Sigma_ju_{ji}^2=1$ (coming from $g_i^2=1$).
\end{enumerate}

From (2) we get $u_{ki}=\sum_ju_{ji}^2u_{ki}=u_{ki}^3$, and together with (1), this shows that the elements $u_{ij}$ are partial isometries. With this observation in hand, the above relations are simply the ``cubic'' relations in \cite{bbc}, for partial isometries, and this gives the result.
\end{proof}

We refer to \cite{bbc} for a full discussion of the representation theory invariants of $H_n^+$, and to
\cite{bb+}, \cite{ez2} for some technical versions and generalizations of these computations.

\begin{proposition}
Theorems 2.5, 3.5, 5.2 hold as well at $s=3$, and give the presentation, diagrams, and spectral measure for the quantum group $G^+(\widehat{\mathbb Z_3^{*n}})$.
\end{proposition}

\begin{proof}
This is clear from a careful examination of the proofs of the above results: the beginning of the proof of Theorem 2.5 holds as well at $s=3$, because here we do not have length 2 relations either, and the rest of the proofs simply hold, unchanged.
\end{proof}

At $s=4$ now, the situation is quite complicated, and we currently do not have general results. Let us just state
and prove the following result, slightly improving a result in \cite{bhs}.

\begin{proposition}
We have a $*$-algebra isomorphism $C(G^+(\widehat{\mathbb Z_4}))\simeq C^*(D_\infty\times\mathbb Z_2)$.
\end{proposition}

\begin{proof}
We use the notations in \cite{bhs}. According to the results there, the fundamental corepresentation of our quantum isometry group must be of the following type:
$$u=\begin{pmatrix}A&B\\ B^*&A^*\end{pmatrix}$$

The generators must satisfy the defining relations for $SU_2^{-1}$, namely:
$$AA^*+BB^*=1,\quad AB+BA=0,\quad AB^*+B^*A=0$$

In addition, the operator $T=A^2+B^2$ must satisfy the following relations:
$$T=T^*,\quad T^2=1,\quad TA=A^*,\quad TB=B^*$$

See \cite{bhs}. The point now is that, since $T$ commutes with $A,B,A^*,B^*$, we can write:
$$T=\begin{pmatrix}1&0\\ 0&-1\end{pmatrix}$$

Thus the quantum isometry algebra in the statement decomposes naturally as a direct sum of two $*$-algebras, and by passing to the generators $A+B,A-B$, we see that each of these $*$-algebras is isomorphic to the group algebra $C^*(D_\infty)$. This gives the result.
\end{proof}

The above result suggests that $G^+(\widehat{\mathbb Z_4})$ might be a twist of the dual of
$D_\infty\times\mathbb Z_2$. This is probably true, but unfortunately we do not see an analogue of this result,
for $n>1$.

Let us discuss now the case $s=\infty$. We have two statements here.

\begin{proposition}
$G^+(\widehat{F}_n)$ is the quantum group $H_{n0}^+$ from \cite{bsk}.
\end{proposition}

\begin{proof}
This result, and the precise relation between $H_{n0}^+$ and $H_n^+$, are explained in \cite{bsk}.
\end{proof}

Our second statement concerns the ``representation-theoretic'' limit, with $s\to\infty$, of the quantum groups $G^+(\widehat{\mathbb Z_s^{*n}})$. Quite surprisingly, this limiting quantum group is not the quantum group $H_{n0}^+$ from \cite{bsk}. Its main properties can be summarized as follows.

\begin{theorem}
Let $K_n$ be the quantum subgroup of $H_{n0}^+$ presented by the relations making all the standard generators $u_{ij}$ normal. Then:
\begin{enumerate}
\item $Hom(u^{\otimes k},u^{\otimes l})=span(T_\pi|\pi\in D_\infty(k,l))$, for any $k,l\in\mathbb N$.

\item $law(\chi/2)=\pi_{\underline{\varepsilon}/2}$, where $\varepsilon$ is the uniform measure on the unit circle.
\end{enumerate}
\end{theorem}

\begin{proof}
We recall from section 3 above that we have inclusions $\overline{D}_\infty(k,l)\subset D_\infty(k,l)$, for any $k,l$. The idea will be to find first a categorical generation result of type $D_\infty=<\overline{D}_\infty,\pi>$, and then to find the quantum group and probabilistic interpretations of this result.

So, let $\pi\in D_\infty(0,4)$ be the 1-block partition colored $bbww$. Our claim is that we have $D_\infty=<\overline{D}_\infty,\pi>$. Since the inclusion $\supset$ is clear, we just have to prove the inclusion $\subset$.

In order to do this, let us look at the category $D=<\overline{D}_\infty,\pi>$. By using Frobenius duality, i.e. by using rotations, we can deduce from $\pi\in D$ the fact that $D$ contains all the 1-block partitions having 0 upper legs and 4 lower legs, with the 4 lower legs colored half-black, half-white. So, in other words, we have $D(0,4)=D_\infty(0,4)$.

Now by using once again Frobenius duality, i.e. by using rotations, we deduce that we have $D(2,2)=D_\infty(2,2)$. In particular, if $\sigma\in D_\infty(2,2)$ denotes the 1-block partition with the upper legs colored $bw$ and the lower legs colored $wb$, then $\sigma\in D(2,2)$.

With this observation in hand, we can finish the proof of the above claim: indeed, by using $\sigma$ we can exchange the position of two consecutive black and white colors in any partition in $D$, and it is easy to see that this operation allows one to construct any element in $D_\infty$ starting from the elements of $\overline{D}_\infty$. So, we have $D_\infty=<\overline{D}_\infty,\pi>$ as claimed.

(1) Let us try to compute the quantum group $K_n'\subset H_{n0}^+$ having as Hom spaces the linear spaces $span(T_\pi|\pi\in D_\infty(k,l))$. Since we have $D_\infty=<\overline{D}_\infty,\pi>$, this quantum group $K_n'\subset H_{n0}^+$ is simply the one presented by the relations coming from $T_\pi\in Fix(u^{\otimes 4})$. Now, according to the definitions in section 3 above, we have:
$$T_\pi=\sum_{ia}e_{ia}\otimes e_{ia}\otimes e_{i\bar{a}}\otimes e_{i\bar{a}}$$

By using the relations in Proposition 2.3, we see that this vector is fixed by $u^{\otimes 4}$ if and only if the following condition is satisfied, where $p_{ia,jb}=u_{ia,jb}u^*_{ia,jb}$, $q_{ia,jb}=u^*_{ia,jb}u_{ia,jb}$:
$$q_{ia,jb}p_{ia,kc}q_{ia,ld}=
\begin{cases}q_{ia,jb}&{\rm if}\  (jb)=(kc)=(ld)\\
0&{\rm otherwise}
\end{cases}$$

This condition is easily seen to be equivalent to the ``normality'' relations $p_{ia,jb}=q_{ia,jb}$, for any $i,j,a,b$, so the quantum group $K_n'$ that we are currently computing is nothing but the quantum group $K_n$ appearing in the statement, and we are done.

(2) This follows from (1), as in the proof of Theorem 5.2, after of course performing some obvious modifications in the preliminary material in sections 4 and 5.
\end{proof}

One question arising from the above result is that of finding a suitable geometric interpretation of $K_n$. Since this is a subgroup of $H_{n0}^+=G^+(\widehat{F}_n)$, we can expect $K_n$ to consist of the ``quantum isometries'' of $\widehat{F}_n$ preserving not only the length, but also some ``extra structure''.  A careful study here leads to the following informal answer: ``under the action of $K_n$ a word in $F_n=\mathbb{Z} * \cdots * \mathbb{Z}$ should get sent to the words of the same length, with the switches between the generators coming from various copies of $\mathbb Z$ appearing at the same places''. This will be explained in detail in a future paper from the present series.

\section{Concluding remarks}

We have seen in this paper that the representation theory invariants of $G^+(\widehat{\Gamma})$ can be explicitly computed in the case $\Gamma=\mathbb Z_s^{*n}$, by using diagrammatic techniques. The answer that we obtain -- an explicit formula for the spectral measure, as a compound free Poisson law $\pi_{\underline{\varepsilon}}$ -- appears to be quite interesting from the point of view of free probability.

There are several questions arising from the present work. The main one is probably the computation of the invariants for general free products of cyclic groups, $\Gamma=\mathbb Z_{s_1}*\ldots*\mathbb Z_{s_n}$. Here we can definitely expect to have some very interesting, new combinatorics, ultimately coming from the arithmetic properties of the sequence of indices $s_1,\ldots,s_n$.

Understanding this next-step combinatorics looks like an important task towards the general understanding of the quantum isometry groups introduced in \cite{gos}, and of the easy quantum groups introduced in \cite{bsp}. We intend to come back to this fundamental question, at least with  partial results, in some future work.

\end{document}